\documentclass[12pt]{article}
\usepackage{latexsym,amsfonts,amsmath,amsthm,amssymb,epsfig}
\usepackage{color}
\usepackage{hyperref}

\usepackage[normalem]{ulem}

\newtheorem{thm}{Theorem}
\newtheorem{lem}{Lemma}
\newtheorem{prop}{Proposition}

\newtheorem*{OP}{Open Problem}

\theoremstyle{remark}
\newtheorem{rem}{Remark}

\theoremstyle{definition}

\newenvironment{kasten}
    {\begin{center}
    \begin{tabular}{|p{0.95\textwidth}|}
    \hline\\
    \vspace*{-5mm}}
    { 
    \vspace*{-3mm}
    \\\\\hline
    \end{tabular} 
    \end{center}
    }

\newcommand{\R}{\mathbb{R}}
\newcommand{\N}{\mathbb{N}}
\newcommand{\C}{\mathbb{C}}
\newcommand{\Z}{\mathbb{Z}}
\newcommand\INT{{\rm INT}}

\title{Lower bounds for integration \\ and recovery in~$L_2$}  

\author{Aicke Hinrichs\footnote{Institut f\"ur Analysis, 
Johannes Kepler Universit\"at Linz, 
Altenbergerstrasse~69, 4040~Linz, Austria, 
Email: \texttt{aicke.hinrichs@jku.at}, \texttt{david.krieg@jku.at}; the research of these authors is supported by the Austrian Science Fund (FWF) Project
F5513-N26, which is a part of the Special Research Program Quasi-Monte Carlo Methods: Theory and Applications.}
,
David Krieg$^*$, 
Erich Novak\footnote{Mathematisches Institut, FSU Jena, 
Ernst-Abbe-Platz 2, 07740 Jena, Germany, 
Email: \texttt{erich.novak@uni-jena.de}.}
, 
Jan Vyb\'iral\footnote{Dept.\ of 
Mathematics FNSPE, Czech Technical University 
in Prague, Trojanova 13, 12000 Prague, Czech Republic, 
Email: \texttt{jan.vybiral@fjfi.cvut.cz}; the research of this author was supported
by the grant P201/18/00580S of the Grant Agency of the Czech Republic and by
the European Regional Development Fund-Project ``Center 
for Advanced Applied Science''
(No. CZ.02.1.01/0.0/0.0/16\_019/0000778).
}}

\begin{document}

\date{\today} 

\maketitle

\begin{abstract} 
Function values are, in some sense,
``almost as good'' as general linear
information for $L_2$-approximation
(optimal recovery, data assimilation) of functions from a reproducing kernel Hilbert space.
This was recently proved by new 
\emph{upper bounds} on the sampling numbers
under the assumption that the singular values of the embedding of this Hilbert space into $L_2$ are square-summable.
Here we mainly prove new \emph{lower bounds}.
In particular we prove that the sampling numbers 
behave worse than the approximation numbers 
for Sobolev spaces with small smoothness. 
Hence there can be a logarithmic gap also 
in the case where the singular numbers of the embedding 
are square-summable.
We first prove new lower bounds  
for the integration problem, again for rather classical Sobolev spaces 
of periodic univariate functions.  
\end{abstract} 

\section{Introduction and Main Results} 

We always assume that $H$ is a separable reproducing kernel 
Hilbert space (RKHS) on a domain $D$ 
and that there is a measure $\mu$ on $D$ 
such that $H$ is compactly embedded into $L_2=L_2(D,\mu)$.
We study algorithms $A_n$ for $L_2$-approximation 
(or optimal recovery) of functions from $H$ 
and consider the worst case error 
\[
e(A_n)= \sup_{\Vert f \Vert_H \le 1} 
\Vert f - A_n(f) \Vert_2.
\] 
We study two kinds of information and algorithms: 
The algorithm 
$A_n(f) = \sum_{i=1}^n L_i (f) g_i$, 
where $g_i \in L_2$, may use arbitrary linear 
functionals $L_i$, while 
$S_n(f) = \sum_{i=1}^n f(x_i) g_i$ 
can only use function values. 
The following question was recently studied in several papers: 
Are sampling algorithms $S_n$  
always, i.e., for any RKHS $H$, 
almost as good as general algorithms $A_n$? 

\medskip

It is well known how to characterize the minimal 
worst case error 
$e(A_n)$ for general algorithms $A_n$. The answer is given by the {\em approximation numbers}
$$
a_n(H,L_2) = \inf_{A_n} e(A_n) = \sigma_{n+1},
$$
where $\sigma_1\ge\sigma_2\ge\dots\ge 0$ are the singular values 
of the compact embedding of $H$ into $L_2$. 
In addition to the approximation numbers (or linear widths) 
$a_n(H,L_2)$ we also define the {\em sampling numbers} 
$$
g_n(H,L_2) = \inf_{S_n} e(S_n)  ,
$$
with algorithms that only use function values.

\medskip

The lively history of \emph{upper bounds} for the sampling numbers $g_n$
for general RKHSs $H$ was initiated by \cite{WW01} and \cite{KWW09},
where the authors assumed that the sequence $\sigma=(\sigma_n)_{n\ge 1}$
of the singular values is in $\ell_2$.
Under this condition it was proved 
in \cite{KU19} that the polynomial 
order of the $a_n$ and the $g_n$
coincides (solving Open Problem 126 
from \cite{NW12}); 
but it was not clear whether a logarithmic 
gap is possible or not. 
Here we show that such a gap is possible
and, in order to do so, prove new 
\emph{lower bounds}. 
We discuss upper bounds on the sampling numbers $g_n$ 
in Remark~\ref{rem:upp},
this paper is mainly on lower bounds.

\medskip

The lower bound
$$
g_n(H,L_2) \ge a_n (H, L_2)
$$
is trivial. 
Although there are several papers which
improve upon this bound in the sense of tractability, see Remark~\ref{rem:curse},
the authors only know one  
paper, namely  \cite{HNV08}, that contains results 
concerning a different
asymptotic behavior
of the $g_n$ and the $a_n$.

\medskip

One way of obtaining nontrivial lower bounds for the numbers $g_n$
is to take a detour and prove lower bounds for the 
problem of numerical integration.
This is the approach we will take in the present paper.
For the necessary notation, let now 
$h\in L_2(D,\mu)$ with $\Vert h \Vert_2 =1$ and consider the functional
\[
{\rm INT}_h(f) 
= \int_D f(x)\, \overline{h(x)} \,{\rm d}\mu(x)
\]
on $H$. 
Let
$Q_n(f) = \sum_{i=1}^n w_i f(x_i) $ 
with scalar $w_i$ be a quadrature formula to approximate 
the integral
and let
$e(Q_n)= \sup_{\Vert f \Vert_{H} \le 1} 
| {\rm INT}_h (f) - Q_n(f) |$ be the worst case error of $Q_n$. 
The minimal worst case error for ${\rm INT}_h$ is then defined by
$$
 e_n(H,{\rm INT}_h ) = \inf_{Q_n} e(Q_n)  
$$
with the infimum taken over all quadrature formulas $Q_n$ using $n$ function values.
Associating with a sampling algorithm $S_n(f) = \sum_{i=1}^n f(x_i) g_i$ 
the quadrature formula 
with $w_i = {\rm INT}_h (g_i)$, we obtain $e(Q_n) \le e(S_n)$ and conclude that
\begin{equation}   \label{eq:bound} 
 g_n(H,L_2) \ge e_n(H,{\rm INT}_h ).
\end{equation}
Furthermore, it was observed in \cite[Proposition 1]{HKNV} that lower bounds
for the integration problem ${\rm INT}_h$ are equivalent to showing that certain matrices
involving the values of the reproducing kernel $K$ and the representer $h$ are positive semi-definite.
To be more specific,
$$
e_n(H,{\rm INT}_h)\ge \|h\|_H^2-\alpha^{-1}
$$
for a real parameter $\alpha>0$ if, and only if,
$$
\{K(x_j,x_k)\}_{j,k=1}^n\succeq \alpha\{h(x_j)h(x_k)\}_{j,k=1}^n
$$
holds for every set of points $\{x_1,\dots,x_n\}\subset D.$ Here, if $A$ and $B$ are symmetric matrices, we mean by $A\succeq B$
that $A-B$ is positive semi-definite. This can be further combined with a recent result of \cite{V},
which shows that the entry-wise product (which is sometimes also called Schur product) of a positive semi-definite matrix
with itself is not only positive semi-definite but also bounded from below by some rank-1 matrix (in the ``$\succeq$'' ordering).
Therefore, if $K$ is a square of another reproducing kernel, the method of \cite{V} and \cite{HKNV} can be
applied
to get
new lower bounds for the integration problem.

\smallskip

We now turn to our results.
It is by now well understood that (upper and lower) 
bounds on the sampling numbers $g_n$ very much depend 
on whether the sequence $\sigma$ of singular numbers 
is square summable or not. 
Equivalent conditions are that
the embedding ${\rm id}$ of $H$ 
into $L_2$ is a Hilbert-Schmidt operator 
or that the operator ${\rm id}^\ast {\rm id}$ has a finite trace.
Lower bounds in \cite{HNV08} 
are proved only for the case $\sigma \notin \ell_2$, 
here we study both cases. 

\medskip

We begin with the assumption $\sigma \in \ell_2$, where we can work 
with classical Sobolev spaces $H_\gamma$ of 
univariate periodic
functions
on $D=[0,1]$.  
Let 
$\gamma=(\gamma_k)_{k\in\Z}$ be a bounded non-negative sequence
and put $e_j(x)=e^{2\pi ijx}$ for $j\in\Z$. Then 
$H_\gamma$ is the set of all 1-periodic functions given by
\begin{equation}\label{eq:Hsigma1}
f(x)=\sum_{j\in\Z} \alpha_je_j(x),\quad x\in[0,1],
\end{equation}
such that $\alpha_j=0$ if $\gamma_j=0$ and
\[
\|f\|_{H_\gamma}^2=\sum_{j:\gamma_j\not=0}\frac{\alpha_j^2}
{\gamma_j^2}<\infty.
\]
It follows from the Cauchy-Schwarz inequality that the series 
in \eqref{eq:Hsigma1} is absolutely and uniformly convergent 
for $\gamma \in \ell_2(\Z)$ and that the point evaluation 
functionals $x \mapsto f(x)$ are continuous. Hence $H_\gamma$ is 
a reproducing kernel Hilbert space of continuous 1-periodic 
functions
with the kernel given by $K(x,y)=\sum_{j\in\Z}\gamma_j^2 e_j(x-y)$
if $\gamma \in \ell_2(\Z)$ and $\alpha_j = \langle f, e_j \rangle_2$ 
is the $j$-th Fourier coefficient of $f$. 
The singular values $\sigma \in \ell_2(\N)$ of $H_\gamma$ in $L_2$ are
given by the non-increasing rearrangement of $\gamma$.

\medskip

Surprisingly, we
find that already for this classical example $H_\gamma$,
there is a gap of order $\sqrt{\log n}$ between the
sampling and the approximation numbers
if we take $\gamma_k$ of order $|k|^{-1/2} \log^{-\beta} |k|$ 
for some $\beta > 1/2$.
These spaces $H_\gamma$ fall into the scale of function 
spaces of generalized (or logarithmic) smoothness,
which can be traced back (at least) to the work of L\'evy \cite{Levy}. 
These spaces were since then studied intensively \cite{DomTik, ET98, Goldman, Moura}.
Although different approximation quantities (like entropy numbers 
or approximation numbers) were studied in the frame of this scale of function spaces earlier,
the possible gap between approximation and sampling numbers 
in the Hilbert space setting went unnoticed so far. %; \change{see also \cite{GS19,TU20,TU21}.} 

\medskip

We prove the logarithmic gap using \eqref{eq:bound} and lower bounds for the problem
of approximating the integral ${\rm INT} (f) = \int_0^1 f(x) \, dx$ for $f \in H_\gamma$.
Note that we omit the $h$ in $ {\rm INT}_h$ in this special case $h=1$.
Of course these bounds are interesting by itself.
Our first main result is the following, see Theorem~\ref{lem:sigma_new}. 

\begin{kasten}
Let $\mu\in\ell_1(\Z)$ 
be a non-negative and non-zero
sequence and let 
\begin{equation*}
\gamma^2_{\ell}=\sum_{j\in\Z}\mu_j\mu_{j+\ell},
\quad \ell\in\Z.
\end{equation*}
Then we have for all $n\in\N_0$ that
\[
e_n(H_{\gamma},{\rm INT})^2
\,\ge\,
\gamma_0^2\left( 1 - \frac{n\gamma_0^2}
{\|\gamma\|_2^2}\right).
\vspace*{-7mm}
\]
\end{kasten}

Using the spaces $H_\gamma$ and the above result one can finally answer 
the following question:  
Is there a Hilbert space $H$ such that the singular 
values of its embedding into $L_2$ are square summable and 
$$
\lim_{n\to \infty}   a_n(H,L_2) / g_n (H,L_2) =0 \, ?
$$
Indeed,
the result
below shows that 
there is a gap of order $\sqrt{\log n}$ between sampling and approximation 
numbers for the spaces $H_\gamma$ with $\gamma$ just in $\ell_2$. 
In this case also the  minimal worst case errors of uniform 
integration, which are a lower bound for the sampling numbers, 
actually have the same asymptotic behavior as the sampling numbers. 
The following is from Theorem~\ref{thm:border-case}.

\begin{kasten}
Let $\beta>1/2$. Then there exists $\gamma\in\ell_2(\Z)$ such that 
\[
 a_n(H_\gamma,L_2) \asymp n^{-1/2} \log^{-\beta} n
\] 
and
\[
e_n(H_{\gamma},{\rm INT})
\,\asymp\,
g_n(H_\gamma,L_2)
\,\asymp\,
n^{-1/2} \log^{-\beta+1/2} n .
\vspace*{-7mm}
\]
\end{kasten}

The symbol $\asymp$ means that the left hand side is bounded 
from above by a constant multiple of the right hand side for 
almost all (i.e., all except finitely many) $n$ and vice versa;
we use $\preccurlyeq$ and $\succcurlyeq$ for the one-sided relations.
%\change{We note that the logarithmic gap was already conjectured in \cite{GS19}
%for quadrature rules with equal weights, so-called quasi-Monte-Carlo methods.
%In Theorem~\ref{thm:border-case}, we prove it for general quadrature rules.}

\medskip

We also obtain a lower bound for arbitrary sequences 
of approximation numbers $(a_n) \in \ell_2(\N_0)$.
This lower bound shows that, in general, 
one cannot expect the sampling numbers~$g_n$ to behave better than 
\[
 \max\bigg\{a_n \, , \bigg(\frac{1}{n} \sum_{k\ge n} a_k^2\,\bigg)^{1/2}\, \bigg\}.
\]
Note that this is conjectured 
to be also the worst possible
behavior of the sampling numbers, see~\cite{KU19}
and Remark~\ref{rem:upp}.
The following is from Theorem~\ref{thm:upperboundisoptimal}.

\begin{kasten}
Let $a \in \ell_2(\N_0)$ be non-increasing.
Then there is some $\gamma \in \ell_2(\Z)$
such that $a_n(H_\gamma,L_2)=a_n$ for all $n\in\N_0$
and
\[
 g_n(H_\gamma,L_2)^2 \,\ge\, \frac{1}{8n} \sum_{k\ge n} a_k(H_\gamma,L_2)^2
\]
for infinitely many values of $n$.
\end{kasten}

Now we present a result 
for the case
$\sigma\not\in\ell_2(\N)$.
Here,
it is already known from \cite{HNV08} that
the convergence of the sampling numbers 
can be extremely slow. 
We improve upon \cite{HNV08} by providing
a lower bound that holds for all $n\in\N$
instead of only a thin sub-sequence.
The following is from Theorem~\ref{thm:infinite}.
For illustration, one might imagine that $a_n \asymp n^{-1/2}$
and $\tau_n\asymp \log^{-1/2}n$.

\begin{kasten}
Let $a, \tau \in c_0(\N_0)$ be non-increasing 
with $a\not\in\ell_2(\N_0)$.
Then there is an example $(H,L_2)$ such that 
$a_n(H,L_2) = a_n$ for all $n\in \N_0$
and
\[
 g_n(H,L_2) 
 \,\ge\,  \tau_n
\]
for almost all values of $n\in\N_0$.
\end{kasten}

We note that also this lower bound
already holds for some kind of integration problem,
which is described in Section~\ref{sec:infinite}.
Here is an open problem that 
also describes some of the progress of this paper.

\begin{OP}
Assume that, for some RKHS $H$, 
\[
 a_n(H,L_2) \asymp n^{-r} \log^{-\beta} n, 
\] 
where $(a_n)_n \in \ell_2(\N_0)$ 
(hence $r>1/2$ or $r=1/2$ and $\beta > 1/2$). 
Does it follow that
\[
 a_n(H,L_2) \asymp  g_n(H,L_2) ? 
\] 
This was posed in \cite{NW12} as a part of the 
Open Problem 140.
We now know the answer ``no'' if $r=1/2$ but do not know the answer 
if $r>1/2$. 
\end{OP}

\begin{rem}[Multivariate analogues]
Theorem~\ref{thm:border-case} answers the above open problem in the case $r=1/2$. 
Nonetheless, it would be interesting to know whether a result like Theorem~\ref{thm:border-case} is also true for Sobolev spaces with small smoothness on the $d$-dimensional sphere or torus for $d>1$; see also \cite{GS19,TU20,TU21}. We hope that our technique can be applied also for this purpose.
\end{rem}

\begin{rem}[Upper bounds]  \label{rem:upp}
Upper bounds for particular spaces have a long history
but it seems that \cite{WW01} and \cite{KWW09} 
are the first papers that study upper bounds for 
general RKHSs. 
The history till 2012 can be found in \cite{NW12}.

This line of study was further developed in \cite{KU19}, where the authors showed that
there are two absolute constants $c,C>0$, such that for every RKHS $H$ and every $n\in\N$
\begin{equation}\label{eq:KU1}
g_n(H,L_2)^2\le \frac{C}{k_n}\sum_{j\ge k_n}a_j(H,L_2)^2
\end{equation}
holds with $k_n\ge c n/\log(n+1).$
It follows as a simple corollary, that in the case $\sigma \in \ell_2$ 
there cannot be a polynomial gap 
between the $a_n$ and the $g_n$.
This solved an open problem from \cite{HNV08}
that was also posed as Open Problem 126 in \cite{NW12}. 
For the solution it was important to understand the 
geometry (and stochastics) of random sections 
of ellipsoids in high dimensional 
euclidean spaces, see \cite{HKNPU19}. We refer to \cite{KUV,KU21,NSU20,Te20,TU20} for improvements of \eqref{eq:KU1}
and further results. Let us also 
remark that one of the aims of our work is to study the optimality of \eqref{eq:KU1}
by providing %the 
appropriate lower bounds, cf.\ Theorem \ref{thm:upperboundisoptimal}.
\end{rem}

\begin{rem}[Tractability and curse of dimensionality] 
\label{rem:curse} 
Assume now that a whole sequence of 
Hilbert spaces $H_d$ is given; 
the functions $f$ from $H_d$ could be defined on $[0,1]^d$. 
For some spaces we know that the curse of 
dimensionality is present, if only function values are allowed, 
while the problem is tractable for general linear information. 
This happens for certain (periodic and nonperiodic) 
Sobolev spaces where the singular values are in $\ell_2$, see  
\cite{NW12,NW16}. For the proof one again uses \eqref{eq:bound} 
together with lower bounds for the integration problem
that follow from the technique of decomposable kernels, see \cite{NW10}.
A new technique to prove intractibility for integration 
problems is based on the method developed in \cite{V}, 
see
\cite{HKNV}. 
This method is also used in the present paper, see 
the proof of Theorem \ref{lem:sigma_new}.
\end{rem} 

\begin{rem}[Randomized algorithms]
\label{rem:rand} 
In this paper we use the worst case setting for 
deterministic algorithms. 
We want to stress that results in the randomized setting 
are quite different. 
In particular, the results do not depend strongly 
on the assumption whether the singular values are 
in $\ell_2$ or not; 
see
\cite{CD21,CM17,Kr19,NW12,WW07}. 
Together with the upper bound from \cite{CD21},
Theorem~\ref{thm:border-case} gives an example
where randomized algorithms achieve a better rate of convergence
for $L_2$-approximation than deterministic algorithms.
\end{rem} 

\section{Finite Trace - Lower Bounds}

In this section we investigate the sampling numbers of the Sobolev spaces $H_\gamma$ of 1-periodic functions defined in the introduction. 
We start with a lower bound based on the results from \cite{HKNV}.
 
\begin{thm}\label{lem:sigma_new}
Let $\mu\in\ell_1(\Z)$ be a non-negative 
sequence, $\mu \ne 0$, 
and let 
\begin{equation}\label{eq:musquare}
\gamma^2_{\ell}=\sum_{j\in\Z}\mu_j\mu_{j+\ell}
\end{equation}
for $\ell \in \Z$.
Then we have for all $n\in\N_0$ that
\begin{equation*}
e_n(H_{\gamma},{\rm INT})^2
\,\ge\,
\gamma_0^2\left( 1 - \frac{n\gamma_0^2}
{\|\gamma\|_2^2}\right).
\end{equation*}
\end{thm}
\begin{proof}
Recall that $e_k(x)=e^{2\pi i k x}$ for $k\in \Z$ and $x\in\R$.
We define the kernels $M: [0,1)^2 \to \C $ and $K: [0,1)^2 \to \R $ by 
\[
M(x,y)=\sum_{k\in\Z}\mu_k e_k(x-y)
\]
and
\[
K(x,y)=\big|M(x,y)\big|^2=\sum_{j,k\in\Z}\mu_j\mu_k 
e_{k-j}(x-y)=\sum_{\ell\in\Z}\gamma_\ell^2 e_\ell(x-y).
\]
We observe that
\[
M(x,x)^2=\left(\sum_{k\in\Z}\mu_k\right)^2=\sum_{\ell\in\Z}\gamma_\ell^2 = \|\gamma\|_2^2
\]
for every $x\in[0,1)$.
Furthermore, the representer of the integration functional $f \mapsto \int_0^1 f(x) \, dx$ 
for $f \in H_{\gamma}$ is given 
by $h=\gamma_0^2e_0$, since for $f$ as in \eqref{eq:Hsigma1}
we have
\[
\int_0^1 f(x)\, dx=\alpha_0=\langle \gamma_0 f,
\gamma_0e_0\rangle_{H_{\gamma}}
=\langle f,h\rangle_{H_{\gamma}}.
\]
In particular, the initial error satisfies
\[
e_0(H_{\gamma},{\rm INT})^2 = \| h \|_{H_{\gamma}}^2 =  \gamma_0^2.
\] 

Let now $n\in \N$ and fix $x_1,\dots,x_n\in[0,1)$.  
Then the matrix $(M(x_j,x_k))_{j,k=1}^n$ is positive 
semi-definite and, by \cite[Theorem 1]{V},
also the matrix with entries
\[
K(x_j,x_k)-\frac{M(x_j,x_j)M(x_k,x_k)}{n} 
\,=\, K(x_j,x_k)-\frac{\|\gamma\|_2^2}
{n\gamma_0^4}h(x_j) h(x_k)
\]
is positive semi-definite.
By \cite[Proposition 1]{HKNV}, it follows that
\[
e_n(H_{\gamma},{\rm INT})^2 \ge 
\|h\|_{H_{\gamma}}^2-\frac{n\gamma_0^4}
{\|\gamma\|_2^2}=\gamma_0^2-
\frac{n\gamma_0^4}{\|\gamma\|_2^2}
\]
as claimed.
\end{proof}

Theorem~\ref{lem:sigma_new} shows that,
if the sequence~$\gamma$ is given by~\eqref{eq:musquare},
we need at least a constant multiple of $\Vert \gamma \Vert_2^2 / \gamma_0^2$ function values
in order to reduce the initial error $\gamma_0$ by a constant multiple.
It would be interesting to know whether this 
is true for all symmetric and non-increasing sequences $\gamma\in\ell_2(\Z)$.
This would simplify the remainder of this section.
Indeed,
assume for a moment 
that for each $n\in\N_0$ we could apply Theorem~\ref{lem:sigma_new} to
the sequence $\tilde \gamma = \gamma_n \mathbf 1_{\{-n,\hdots,n\}} + \gamma \mathbf 1_{\N\setminus\{-n,\hdots,n\}}$.
%$\tilde \gamma=(\tilde\gamma_k)_{k\in\Z}$, where 
%$\tilde\gamma_k=\gamma_n$ if $|k|\le n$ and $\tilde\gamma_k=\gamma_k$ if $|k|>n$,
%\[
%\tilde\gamma_k=\begin{cases} \gamma_n \quad &\text{if}\ |k|\le n,\\
%\gamma_k\quad &\text{if}\ |k|>n
%\end{cases}
%\]
Then we obtain for every $m\in\N_0$ that
\[
e_m(H_{\gamma},\INT)^2\ge e_m(H_{\tilde \gamma},\INT)^2\ge \gamma_n^2\left(1-\frac{m\gamma_n^2}{(2n+1)\gamma_n^2+2\sum_{k>n}\gamma_k^2}\right).
\]
%Essentially, %By monotonicity,
This would imply that we need at least
\[
m(n) \,:=\, \Big\lceil  n+\sum\nolimits_{k> n} \gamma_k^2 /\gamma_n^2 \Big\rceil
\]
function values
to achieve a squared error smaller than $\gamma_n^2/2$.
%and simplify the remainder of this section. \change{Indeed, 
Thus we would get for all $m=m(n)$, $n\in\N_0$, that
\[
e_m(H_{\gamma},\INT)^2\,\ge\, \frac{\gamma_n^2}{2} 
\,\ge\, \frac{1}{2m}\sum_{k> n}\gamma_k^2
\,\ge\, \frac{1}{2m}\sum_{k> m}\gamma_k^2.
\]
In particular, this would imply Theorem~\ref{thm:upperboundisoptimal}.
However, since Theorem~\ref{lem:sigma_new} is only for sequences of the form~\eqref{eq:musquare},
we need to do some additional work to get there.
We approximate general sequences $\gamma$ by sequences of the form~\eqref{eq:musquare}.

\begin{lem}\label{lem:main}
Let $\gamma \in \ell_2(\Z)$ be non-negative and non-increasing on $\N_0$ and 
let $\gamma_{-k} \ge \gamma_k$ for all $k\in\N$.
For $r\in \N_0$, we put
\begin{equation}\label{eq:nr}
n(r) \,:=\, \left\lfloor \frac{\left( \sum_{j\ge r} 
\gamma_j^2 \right)^2}{2 \sum_{j\ge r} \gamma_j^4}  \right\rfloor\,.
\end{equation}
Then we have
\[
e_{n(r)}(H_{\gamma},{\rm INT})^2
\,\ge\,
\frac{1}{2(n(r)+1)} \sum_{j\ge r} \gamma_j^2\,.
\]
\end{lem}

\begin{proof} For $k\in\Z$ we define
\[
\mu_k=\begin{cases}0&\text{if}\ k<r,\\
\displaystyle \gamma_k^2\biggl(\sum_{\ell=r}^\infty 
\gamma_\ell^2\biggr)^{-1/2}&\text{if}\ k\ge r.\end{cases}
\]
We associate to $\mu=(\mu_k)_{k\in\Z}$ the 
sequence $\tilde\gamma=(\tilde\gamma_\ell)_{l\in\Z}$ by 
\eqref{eq:musquare} and observe that
\[
\tilde\gamma_{\ell}^2=\sum_{j=r}^\infty 
\gamma_j^2\gamma_{\ell+j}^2\Bigl(\sum_{u=r}^\infty
\gamma_u^2\Bigr)^{-1}\le \gamma_{\ell}^2
\sum_{j=r}^\infty \gamma_j^2\cdot 
\Bigl(\sum_{u=r}^\infty\gamma_u^2\Bigr)^{-1}=\gamma_\ell^2
\]
for $\ell\ge 0$ and
\[
\tilde\gamma^2_{\ell}=\sum_{j\in\Z}\mu_j\mu_{j+\ell}=\sum_{k\in\Z}\mu_{k-\ell}\mu_k=\tilde\gamma^2_{-\ell}\le \gamma_{-\ell}^2
\le \gamma_\ell^2
\]
for $\ell<0.$ Therefore, we have $\tilde \gamma_\ell^2\le \gamma_\ell^2$ for all $\ell\in\Z.$
By monotonicity and Theorem~\ref{lem:sigma_new}, we obtain for all $n\in\N_0$ that
\[
e_n(H_{\gamma},{\rm INT})^2\ge e_n(H_{\tilde\gamma},{\rm INT})^2 
\ge \tilde\gamma_0^2\left(1-\frac{n\tilde\gamma_0^2}{\|\tilde\gamma\|_2^2}\right).
\]
If we put
\[
 n(r) \,:=\, \left\lfloor\frac{\Vert \tilde \gamma \Vert_2^2}{2\,\tilde\gamma_0^2}\right\rfloor\,,
\]
we obtain
\[
e_{n(r)}(H_{\gamma},{\rm INT})^2
\,\ge\, \frac{\tilde\gamma_0^2}{2}
\,\ge\, \frac{\Vert \tilde \gamma \Vert_2^2}{2(n(r)+1)}.
\]
It now only remains to observe that
\[
\tilde\gamma_0^2 
= \sum_{j\in\Z} \mu_j^2
= \frac{\sum_{k\ge r} \gamma_k^4}{\sum_{k\ge r} \gamma_k^2}
\]
and
\[
\Vert\tilde\gamma\Vert_2^2 
= \Bigl(\sum_{j\in\Z}\mu_j\Bigr)^2
= \sum_{k\ge r} \gamma_k^2. 
\]
\end{proof}

From this we get some nice consequences.

\begin{thm}\label{thm:upperboundisoptimal}
Let $a \in \ell_2(\N_0)$ be non-negative and non-increasing.
If we put 
%$\gamma_k = a_{2k}$ for $k\ge 0$ and $\gamma_{-k} = a_{2k-1}$ for $k>0$, 
$\gamma=(...,a_3,a_1,a_0,a_2,a_4,...)$, 
we have
%Then there is some $\gamma \in \ell_2(\Z)$ such that 
$a_n(H_\gamma,L_2)=a_n$ for all $n\in\N_0$
and
\[
 g_n(H_\gamma,L_2)^2 \,\ge\, \frac{1}{8n} \sum_{k\ge n} a_k(H_{\gamma},L_2)^2
\]
for infinitely many values of $n$.
\end{thm}

\begin{proof}
%We consider the Hilbert space $H=H_\gamma$,
%where $\gamma_k := a_{2k}$ for $k\ge 0$ and $\gamma_{-k} := a_{2k-1}$ for $k>0$.
For $r\in\N_0$, let again $n(r)$ be defined by \eqref{eq:nr}.
We distinguish two cases.
In the first case, we assume $n(r) \ge 2r$ for infinitely many $r\in \N$.
For these values of $r$, we get from Lemma~\ref{lem:main} that
\begin{multline*}
g_{n(r)}(H_\gamma,L_2)^2
\,\ge\,
e_{n(r)}(H_\gamma,{\rm INT})^2
\,\ge\,
\frac{1}{2(n(r)+1)} \sum_{j\ge r} \gamma_j^2
\,\ge\,
\frac{1}{4 n(r)} \sum_{j\ge r} a_{2j}^2 \\
\,\ge\,
\frac{1}{8 n(r)} \sum_{j\ge r} \big( a_{2j}^2 + a_{2j+1}^2 \big)
\,\ge\,
\frac{1}{8 n(r)} \sum_{k\ge n(r)} a_k^2
\end{multline*}
and we are done because in this case, the sequence $(n(r))_{r\in\N}$ is unbounded.
\medskip

In the second case, we assume $n(r) \le 2r$ for infinitely many $r\in \N$.
This means that
\[
 2r \,\ge\, 
 \Bigg\lfloor \frac{\left( \sum_{j\ge r} \gamma_j^2 \right)^2}{2 \sum_{j\ge r} \gamma_j^4}  \Bigg\rfloor
 \,\ge\, 
 \Bigg\lfloor \frac{\sum_{j\ge r} \gamma_j^2 }{2 \gamma_r^2}  \Bigg\rfloor
\]
and thus
\[
 2r \,\ge\, 
 \frac{\sum_{j\ge r} \gamma_j^2 }{4 \gamma_r^2} .
\]
Here we estimate
\[
g_{2r}(H_\gamma,L_2)^2
\,\ge\,
a_{2r}(H_\gamma,L_2)^2
\,=\, \gamma_r^2
\,\ge\, \frac{1}{8r} \sum_{j\ge r} \gamma_j^2
\,\ge\, \frac{1}{16r} \sum_{k\ge 2r} a_k^2
\]
to obtain the desired statement.
\end{proof}

A lower bound of this type can be obtained for all $n\in\N$
(instead of just infinitely many)
if the sequence of singular values has some additional regularity.
For example, we get the following.

\begin{thm}\label{thm:upperboundisoptimalplus}
Let $a \in \ell_2(\N_0)$ be non-negative, non-increasing and 
assume that there is a constant $b>0$ such that $a_{2n}\ge b a_n$
for all $n\in\N_0$.
%Then there is 
%some $\gamma \in \ell_2(\Z)$
%such that 
If we put 
$\gamma=(...,a_3,a_1,a_0,a_2,a_4,...)$, 
we have 
$a_n(H_\gamma,L_2)=a_n$ for all $n\in\N_0$
and
\[
 g_n(H_\gamma,L_2)^2 
 \,\ge\, 
 e_n(H_\gamma,{\rm INT})^2 
 \,\succcurlyeq\, \frac{1}{n} \sum_{k\ge n} a_k(H_{\gamma},L_2)^2.
\]
\end{thm}

This theorem is an immediate consequence of the following lemma.

\begin{lem}\label{prop}
Let $\gamma \in \ell_2(\Z)$ be non-negative, non-increasing on $\N_0$ and let $\gamma_{-k}\ge \gamma_k$ for all $k\in\N$.
If there is a constant $b>0$ such that $\gamma_{2k} \ge b \gamma_k$ for all $k\in \N$,
then we have
\[
e_n(H_{\gamma},{\rm INT})^2
\,\succcurlyeq\,
\frac{1}{n} \sum_{j\ge n} \gamma_j^2\,.
\]
\end{lem}

\begin{proof}
We use Lemma~\ref{lem:main} and observe that
\[
n(r) \,\ge\, \frac{\left( \sum_{j\ge r} \gamma_j^2 \right)^2}{2 \sum_{j\ge r} \gamma_j^4} - 1
\,\ge\, \frac{\sum_{j\ge r} \gamma_j^2 }{2 \gamma_r^2} - 1
\,\ge\, \frac{r \gamma_{2r}^2 }{2 \gamma_r^2} - 1
\,\ge\, \frac{b^2 r}{2} - 1.
\]
In particular, $n(r) \to \infty$ and $n(r) \ge 1$ for $r\ge r_0$.
On the other hand, we have
\[
 \sum_{j\ge 2r} \gamma_j^4 
 \,\ge\, \sum_{j\ge r} \gamma_{2 j}^4 
  \,\ge\, b^4 \sum_{j\ge r} \gamma_j^4
\]
and thus for all $r \ge r_0$ that
\[
 n(2r) \,\le\, \frac{\left( \sum_{j\ge 2r} \gamma_j^2 \right)^2}{2 \sum_{j\ge 2r} \gamma_j^4}
 \,\le\, \frac{\left( \sum_{j\ge r} \gamma_j^2 \right)^2}{2 b^4 \sum_{j\ge r} \gamma_j^4}
 \,\le\, \frac{2}{b^4}\, n(r).
\]
Hence, there is a constant $C\in\N$
such that for all $n\in \N$ we find some $r\in\N$
with $n \le n(r) \le Cn$.
We get 
\begin{multline*}
e_{n}(H_{\gamma},{\rm INT})^2
\,\ge\,
e_{n(r)}(H_{\gamma},{\rm INT})^2
\,\ge\,
\frac{1}{2(n(r)+1)} \sum_{j\ge r} \gamma_j^2
\,\ge\,
\frac{1}{4Cn} \sum_{j\ge 4Cn/b^2} \gamma_j^2.
\end{multline*}
Now, choosing $t\in\N$ with $2^t\ge 4C/b^2$,
we continue
\[
 \sum_{j\ge 4Cn/b^2} \gamma_j^2 
 \,\ge\, \sum_{j\ge n} \gamma_{2^t j}^2 
  \,\ge\, b^{2t} \sum_{j\ge n} \gamma_j^2 
\]
and obtain the statement.
\end{proof}

%\begin{proof}[Proof of Theorem~\ref{thm:upperboundisoptimalplus}]
%Again, we choose $H=H_\gamma$ 
%with $\gamma_k := a_{2k}$ for $k\ge 0$ and $\gamma_{-k} := a_{2k-1}$ for $k>0$
%and apply Lemma~\ref{prop}.
%Note that $\gamma_{2k} \ge b \gamma_k$ for all $k\in \N$.
%\end{proof}

\section{Finite Trace - Upper Bounds}

We now complement the lower bound of Lemma \ref{prop} with an appropriate upper bound. 
Let us recall that the lower bounds of Lemma \ref{prop} were based on
a new technique from \cite{HKNV} and \cite{V}.
Compared to that, 
the upper bounds of Propostion~\ref{prop:error_dirichlet_approx} 
and Theorem~\ref{thm:border-case}
are are based on a classical approximation scheme using the Dirichlet kernel,
see, e.g., \cite[Theorem~3.3]{PPST18}.

\medskip

%
%We decided to use the approximation of $f$ by piecewice constant functions $S_n(f)$ to make the presentation essentially self-contained.
%On the other hand, the limited smoothness of $S_n(f)$ limits the use of this method only to function spaces with small smoothness (i.e., to the case we are interested in).
Here, we approximate $f$ by
\[
S_n(f)\,:=\,\frac{1}{2n+1}\sum_{j=0}^{2n} f\left(x^n_j\right)D_n(\,\cdot\,-x^n_j),
\]
where $x^n_j=\frac{j}{2n+1}$ and $D_n(x)=\sum_{|l|\le n}e_l(x)$ is the Dirichlet kernel of degree~$n$.
The integral of $f$ is approximated by the midpoint rule
\[
Q_n(f)\,:=\,\frac{1}{2n+1}\sum_{j=0}^{2n} f\left(x^n_j\right).
\]
Note that $Q_n(f)$ is the integral of $S_n(f)$.

\begin{lem}\label{lem:error_dirichlet_approx}
 Let $f=\sum_{m\in\Z}\alpha_m e_m$ be a 1-periodic function pointwise represented by its Fourier series. 
 Then, for every $n\in\N$,
 \begin{equation*} 
  \| f- S_n(f) \|_2^2 \,=\, 
  \sum_{|j|>n} \alpha_j^2+\sum_{|k|\le n}\Biggl|\sum_{\theta\in\Z\setminus\{0\}}\alpha_{k+\theta(2n+1)}\Biggr|^2.
 \end{equation*}
\end{lem}

\begin{proof}
We calculate the $k$-th Fourier coefficient of $S_n(f)$ by
\begin{align*}
\int_{0}^1 S_n(f)(x)\,\overline{e_k(x)}\,{\rm d}x
&=\frac{1}{2n+1}\sum_{j=0}^{2n} f\left(x^n_j\right)\sum_{l=-n}^n e_{-l}(x^n_j) \int_0^1 e_{l-k}(x)\,{\rm d}x.
%=\frac{1}{2n+1}\sum_{j=0}^{2n} f\left(x^n_j\right) e_{-k}(x^n_j).
\end{align*}
The last expression vanishes if $|k|>n$ and for $|k|\le n$, using the Fourier expansion of 
$f$ %\left(x^n_j\right)=\sum_{m\in\Z}\alpha_{m}e_m\left(x^n_j\right)$
at $x^n_j$,
it is equal to
\[
\sum_{m\in\Z}\alpha_m \cdot\frac{1}{2n+1}\sum_{j=0}^{2n} e_m\left(x^n_j\right) e_{-k}(x^n_j)=\sum_{\theta\in\Z}\alpha_{k+\theta (2n+1)}.
\]
If we make use of the partial sum operator $T_n(f)=\sum_{|k|\le n}\alpha_ke_k$, the calculation above shows that $S_n(T_n(f))=T_n(f)$ and we can compute 
%for $f\in H_\gamma$ with $\|f\|_\gamma\le 1$
\begin{align*}
\|f-S_n(f)\|_2^2&=\|f-T_n(f)+S_n(T_n(f)-f)\|_2^2\\&=\|f-T_n(f)\|_2^2+\|S_n(T_n(f)-f)\|_2^2\\
&= \sum_{|j|>n} \alpha_j^2+\sum_{|k|\le n}\Biggl|\sum_{\theta\in\Z\setminus\{0\}}\alpha_{k+\theta(2n+1)}\Biggr|^2
\end{align*}
to obtain the desired identity.
\end{proof}

From this, one obtains the following general upper bound.

\begin{prop}\label{prop:error_dirichlet_approx}
 Let $\gamma \in \ell_2(\Z)$ be symmetric and non-increasing on $\N_0$.
 Then, for all $n\in\N$, we have
\[
	e(Q_n,H_\gamma,\INT) \,\le\, e(S_n,H_\gamma,L_2) \,\le\, 
	2\,\max\left\{\gamma_{n+1}, \bigg(\frac1n \sum_{k>n}\gamma^2_k\bigg)^{1/2} \right\}.
\]
\end{prop}

\begin{proof}
The first inequality is clear from $Q_n(f)=\INT(S_n(f))$ since
\[
 \left| \INT(f) - Q_n(f) \right|
 = \left| \INT(f - S_n(f)) \right|
 \le \Vert f- S_n(f) \Vert_2.
\]
Regarding the second inequality, Lemma~\ref{lem:error_dirichlet_approx}
yields for $f=\sum_{m\in\Z}\alpha_m e_m$ with $\Vert f\Vert_{H_\gamma}\le 1$ that
\begin{align*}
\| f- S_n&(f) \|_2^2 \,=\, 
  \sum_{|j|>n} \alpha_j^2+\sum_{|k|\le n}\biggl|\sum_{\theta\in\Z\setminus\{0\}}\alpha_{k+\theta(2n+1)}\biggr|^2\\
&\le \sum_{|j|>n} \frac{\alpha_j^2}{\gamma_j^2}\cdot \gamma_j^2+\sum_{|k|\le n} \Bigl\{\sum_{\theta\in\Z\setminus\{0\}}\frac{\alpha^2_{k+\theta(2n+1)}}{\gamma^2_{k+\theta(2n+1)}}\cdot \sum_{\varphi\in\Z\setminus\{0\}}\gamma^2_{k+\varphi(2n+1)}\Bigr\}\\
&\le \gamma_{n+1}^2+\max_{|k|\le n} \sum_{\varphi\in\Z\setminus\{0\}}\gamma^2_{k+\varphi(2n+1)}\cdot \sum_{|k|\le n}\sum_{\theta\in\Z\setminus\{0\}}\frac{\alpha^2_{k+\theta(2n+1)}}{\gamma^2_{k+\theta(2n+1)}}\\
&\le \gamma_{n+1}^2+\max_{|k|\le n} \sum_{\varphi\in\Z\setminus\{0\}}\gamma^2_{k+\varphi(2n+1)}= \gamma_{n+1}^2+\max_{0\le k\le n} \sum_{\varphi\in\Z\setminus\{0\}}\gamma^2_{k+\varphi(2n+1)}\\
&\le \gamma_{n+1}^2+\sum_{l=1}^\infty \gamma^2_{l(2n+1)}+\sum_{l=1}^\infty \gamma^2_{l(2n+1)-n}\\
%&\le \gamma_{n+1}^2+ 2 \sum_{\ell=0}^\infty\gamma^2_{n+1+\ell(2n+1)}
&\,\le\, 2\gamma_{n+1}^2+\frac{1}{n} \sum_{k>n}\gamma^2_k,
\end{align*}
and thus the stated inequality follows.
\end{proof}

Recall that Theorem~\ref{thm:upperboundisoptimal}
gives a lower bound for the recovery problem
that matches the upper bound of Proposition~\ref{prop:error_dirichlet_approx} up to a constant factor
for infinitely many values of $n\in\N$.
In this sense, the results show that there is no significant improvement
over the recovery method $S_n$ 
for any of the Sobolev spaces $H_\gamma$.
%of univariate periodic functions.
If $\gamma$ has some additional regularity, 
we can even say a little more.
Then both $S_n$ and $Q_n$ are optimal for all $n\in\N$
up to constants.

\begin{thm}\label{thm:border-case}
Let $\gamma \in \ell_2(\Z)$ be symmetric and non-increasing on $\N_0$
and assume that there is a constant $b>0$ such that $\gamma_{2n}\ge b \gamma_n$
for all $n\in\N_0$.
Then
\[
e_n(H_{\gamma},{\rm INT})^2
\,\asymp\,
g_n(H_\gamma,L_2)^2
\,\asymp\,
\frac{1}{n} \sum_{j> n} \gamma_j^2\,.
\]
In particular, if we have $\gamma_k = k^{-1/2} \log^{-\beta} k$ for some $\beta > 1/2$
and all $k\ge k_0$, then 
\[
 a_n(H_\gamma,L_2) \asymp n^{-1/2} \log^{-\beta} n
\] 
and
\[
e_n(H_{\gamma},{\rm INT})
\,\asymp\,
g_n(H_\gamma,L_2)
\,\asymp\,
n^{-1/2} \log^{-\beta+1/2} n\,.
\]
\end{thm}

\medskip

\begin{proof}
%Theorem~\ref{thm:border-case} is an 
The lower bound is already stated in Lemma~\ref{prop}.
On the other hand, Proposition~\ref{prop:error_dirichlet_approx} yields the upper bound
\[
	e_n(H_{\gamma},{\rm INT})^2 \,\le\, g_n(H_\gamma,L_2)^2 \,\le\, 
	4\,\max\left\{\gamma_{m+1}^2, \frac1m \sum_{k>m}\gamma^2_k \right\}
\]
with $m=\lfloor \frac{n-1}{2} \rfloor$.
Because of the regularity assumption for $\gamma$,
both $\gamma_{m+1}^2$ and  $\frac1m \sum_{k>m}\gamma^2_k $
are dominated by $\frac{1}{n} \sum_{j> n} \gamma_j^2$.
\end{proof}

\begin{rem}
%This means
By Theorem~\ref{thm:border-case}
we have a gap of order $\sqrt{\log(n)}$ between sampling and approximation numbers
if $\gamma_n \asymp n^{-1/2} \log^{-\beta} n$ with $\beta>1/2$.
%(and symmetric and non-increasing)
%zB $\gamma_k = (1+|k|)^{-1/2} (\log(1+|k|))^{-\beta}$
If we consider $\beta=1/2$ and add a double logarithm of order bigger than $1/2$, 
we obtain a gap of order $\sqrt{\log(n) \cdot \log(\log(n))}$,
and so forth. This is in a sharp contrast to the situation where 
$\gamma_n\asymp n^{-s}$
%$\gamma_k=|k+1|^{-s}$ 
with $s>1/2$.
In that case, it is known (cf.\ \cite[Theorems 2.1.2 and 2.2.1]{Kud98}) that
\[
e_n(H_{\gamma},{\rm INT})\,\asymp\,
g_n(H_\gamma,L_2)\,\asymp\,
a_n(H_\gamma,L_2)\,\asymp\,n^{-s} .
\]
\end{rem}
\smallskip

\begin{rem}
In the case of small smoothness,
%e.g., under the assumption that there is some $k_0\in\N$ such that $(k \gamma_k)_{k\ge k_0}$ is increasing
where we observe the logarithmic gap
between the sampling and approximation numbers,
an upper bound like in Proposition~\ref{prop:error_dirichlet_approx} 
can also be proven for a piecewise constant approximation.
This was done in an earlier version of this manuscript;
see also \cite[Section 2 of Chapter 12]{DL1993}.
Following the advice of a referee,
we replaced this approach 
by the more general approach of Dirichlet approximation, 
which works for any sequence $\gamma$.
\end{rem}

\section{Infinite Trace}
\label{sec:infinite}

In this section, we consider $\sigma\not\in\ell_2(\N)$ and want to show
that there exists a reproducing kernel Hilbert space $H$
whose singular values in $L_2$ are given by $\sigma$ 
and whose sampling numbers $g_n(H,L_2)$ show an arbitrarily bad behavior.
We cannot use the spaces $H_\gamma$ from the previous sections in this case
since those are not reproducing kernel Hilbert spaces any more.
We need different examples.
Here, we consider (real) sequence spaces $H\subset \ell_2(\N)$, 
which are reproducing kernel Hilbert spaces on the domain $D=\N$.
An integration problem ${\rm INT}_h$ for $h\in\ell_2$ therefore takes the form
\[
 {\rm INT}_h(f)=\sum_{j=1}^\infty h_j f_j.
\]
In \cite{HNV08}, Hinrichs, Novak and Vyb\'iral proved the following.

\begin{lem}[\cite{HNV08}]\label{lem:HNV}
 Let $\sigma_1\ge \sigma_2\ge \hdots \ge 0$ such that $\sum_{j=1}^\infty \sigma_j^2 = \infty$
 and
 let $n_0\in\N$ and $\varepsilon>0$.
 Then there is some $m\in \N$ 
 and a Hilbert space $H \subset \R^m$
 as well as some $h\in \R^m$ with $\Vert h \Vert_2 =1$ 
 such that $a_n(H,\ell_2^m) = \sigma_{n+1}$ for all $n < m$ and
\[
 e_{n_0}(H,{\rm INT}_h) \,\ge\, (1-\varepsilon) \sigma_1 .
\]
\end{lem}

\begin{proof}
We note that in \cite[Theorem~1]{HNV08}
only $\sigma_1=1$ is considered but this is just a matter of scaling. 
Moreover, it is only written that $a_n(H,\ell_2^m) = \sigma_{n+1}$
for $n=n_0$ and that 
\[
 g_{n_0}(H,\ell_2^m) \,\ge\, (1-\varepsilon) \sigma_1.
\]
But a second look quickly shows
that the authors prove precisely the stated lemma
for $h=e_1$ and $e_1$ as in \cite{HNV08}. 
\end{proof}

From this, they concluded the following.

\begin{thm}[\cite{HNV08}]
 Let $\sigma_1\ge \sigma_2\ge \hdots \ge 0$ such that $\sum_{j=1}^\infty \sigma_j^2 = \infty$
 and $\tau_0 \ge \tau_1 \ge \hdots \ge 0$ such that $\lim_{n\to\infty} \tau_n =0$.
 Then there is 
 a Hilbert space $H\subset \ell_2(\N)$
 such that $a_n(H,\ell_2) = \sigma_{n+1}$ for all $n\in \N_0$ and
\[
 g_{n}(H,L_2) \,\ge\, \tau_n
\]
for infinitely many values of $n\in\N$.
\end{thm}

Thus, in the case of an infinite trace, 
there are no reasonable upper bounds for the sampling numbers
in terms of the approximation numbers.
At least there are none that hold for (almost) all values of $n$.
For example, it may happen that $a_n(H,L_2) = n^{-1/2}$
for all $n\in\N$, but $g_n(H,L_2) \ge \log^{-1/2} n$ for infinitely many
values of $n\in\N$.

\medskip

But the theorem still leaves us with some hope.
On the one hand, 
it leaves room for upper bounds on the sampling numbers that hold
for infinitely many values of $n\in\N$.
On the other hand, there might still be upper bounds 
for the simpler
problem of computing 
an integral.
In both cases, the hope is not justified.
We prove the following.

\begin{thm}\label{thm:infinite}
 Let $\sigma_1\ge \sigma_2\ge \hdots \ge 0$ such that $\sum_{j=1}^\infty \sigma_j^2 = \infty$
 and $\tau_0 \ge \tau_1 \ge \hdots \ge 0$ such that $\lim_{n\to\infty} \tau_n =0$.
 Then there is 
 a Hilbert space $H\subset \ell_2(\N)$
 such that $a_n(H,\ell_2) = \sigma_{n+1}$ for all $n\in \N_0$
 and some $h\in \ell_2$ with $\Vert h\Vert_2=1$
 such that
\[
 g_n(H,\ell_2) \,\ge\, e_n(H,{\rm INT}_h) \,\ge\,  \tau_n
\]
for almost all values of $n\in\N$.
\end{thm}

For the proof, we first add a slight modification of Lemma~\ref{lem:HNV}.

\begin{lem}\label{lem:HNVmod}
  Let $\sigma_1\ge \sigma_2\ge \hdots \ge 0$ such that $\sum_{j=1}^\infty \sigma_j^2 = \infty$
 and let $n_0\in\N$ and $\varepsilon>0$.
 Then there is a Hilbert space $H \subset \ell_2(\N)$
 as well as some $h\in \ell_2$ with $\Vert h \Vert_2 =1$ 
 such that $a_n(H,\ell_2) = \sigma_{n+1}$ for all $n \in\N_0$ and
\[
 e_{n_0}(H,{\rm INT}_h) \,\ge\, (1-\varepsilon) \sigma_1 .
\]
\end{lem}

\begin{proof}
We take the Hilbert space 
$H_1 \subset \R^m$ and $h\in \R^m \subset \ell_2$ from Lemma~\ref{lem:HNV}
and a Hilbert space $H_2$ that is contained in the orthogonal complement
of $\R^m$ in $\ell_2$ and has singular values $(\sigma_k)_{k>m}$.
We choose $H$ as the direct sum of $H_1$ and $H_2$ in $\ell_2$.
Then the approximation numbers of $H$ in $\ell_2$ are given by $\sigma$.
Moreover, the lower bound of Lemma~\ref{lem:HNV}
extends to $e_{n_0}(H,{\rm INT}_h)$, since the unit ball of $H$ is larger
than the unit ball of $H_1$ and since the additional point evaluations
$f\mapsto f_k$ for $k>m$ that are gained from replacing $\ell_2^m$
by $\ell_2$ are equal to the zero functional on $H_1$. 
\end{proof}

\begin{proof}[Proof of Theorem~\ref{thm:infinite}]
First, we partition $\N$ into index sets $I_j$, $j\in\N$,
such that $I_j$ starts with $2j-1$ and such that the square sum
of $\sigma$ over each index set $I_j$ is still infinite.
For $I_j$, in addition to the odd index $2j-1$, we take every other of the even indices which have not been used yet for $I_i$ with $i<j$. Namely,
\[
 I_j \,=\, \{2j-1\} \cup \{ k\in\N \colon k \equiv 2^j \,{\rm mod}\, 2^{j+1}\}.
\]
Then, because of monotonicity,
\[
 \sum_{k\in I_j} \sigma_k^2 \,\ge\, \sum_{l =1}^\infty \sigma_{l 2^{j+1}}^2
 \,\ge\, \frac{1}{2^{j+1}} \sum_{k=2^{j+1}}^\infty \sigma_k^2 \,=\, \infty.
\]

\medskip

Secondly, we choose natural numbers $n_0<n_1<n_2<\hdots$
such that we have for all $j\in\N$ that
\[
 \tau(n_{j-1}) \le  2^{-j/2} \frac{\sigma_{2j-1}}{2}.
\]
Then, by Lemma~\ref{lem:HNVmod}, there is an example $(H_j,\ell_2(I_j))$
and some $h_j\in \ell_2(I_j)$ with $\Vert h_j \Vert_2 =1$
such that the sequence of singular numbers
is given by $(\sigma_k)_{k\in I_j}$ and
\[
 e_{n_j}(H_j,{\rm INT}_{h_j}) \,\ge\, \frac{\sigma_{2j-1}}{2}.
\]
We define $H \subset \ell_2(\N)$ as the direct sum of the spaces $H_j$.
Namely, $H$ contains all $f\in\ell_2(\N)$
for which $(f_k)_{k\in I_j} \in H_j$ for all $j\in\N$ 
and for which
\[
 \Vert f \Vert_H \,:=\, \bigg(\sum_{j\in\N} \left\Vert (f_k)_{k\in I_j} \right\Vert_{H_j}^2\bigg)^{1/2}
\]
is finite.
Then the sequence of singular numbers of $H$ in $\ell_2(\N)$ is the sequence~$\sigma$.
We put
\[
 h \,:=\, \sum_{j=1}^\infty  2^{-j/2} h_j,
\]
which satisfies $\Vert h \Vert_2 =1$.
Let $n\ge n_0$ and choose $j\in\N$ such that  $n_{j-1} \le n < n_j$.
Then
\begin{multline*}
e_n(H,{\rm INT}_h)  \,\ge\, e_n\left(H_j,{\rm INT}_h\right)
\,=\,  e_n\left(H_j,2^{-j/2}\, {\rm INT}_{h_j}\right)  \\
\,=\,  2^{-j/2} e_n\left(H_j,{\rm INT}_{h_j}\right) 
\,\ge\,  2^{-j/2} \frac{\sigma_{2j-1}}{2}
\,\ge\,  \tau(n_{j-1})  
\,\ge\,  \tau_n,
\end{multline*}
as claimed.
\end{proof}

\medskip 
\noindent
{\bf Acknowledgments: } \
We thank Mathias Sonnleitner, Mario Ullrich, Tino Ullrich, and two anonymous referees
for valuable comments
that helped to improve this manuscript.

\end{document}